\documentclass{amsart}
\usepackage{amssymb,amsmath,epsfig,mathrsfs, enumerate}
\usepackage{graphicx}
\usepackage{fancyhdr}
\usepackage{color,soul}
\fancyheadoffset{0cm}
\usepackage[margin=3cm]{geometry}
\usepackage{hyperref}
\hypersetup{
 colorlinks   = true,
 urlcolor     = blue,
 citecolor   = red ,
 bookmarksopen=true
}

\newtheorem{theorem}{Theorem}[section]
\theoremstyle{plain}

\newtheorem{definition}[theorem]{Definition}

\numberwithin{equation}{section}
\title[]{Existence and Multiplicity results for Weakly coupled system of Pucci's extremal operator.}
\author{Karan Rathore}
\address[Karan Rathore]{VNIT, Nagpur, India-440010} \email{rathorekaran020@gmail.com, ds23mth004@students.vnit.ac.in}
\author{Mohan Mallick}
\address[Mohan Mallick]{VNIT, Nagpur, India-440010} \email{mohan.math09@gmail.com, mohanmallick@mth.vnit.ac.in}
\date{}
\begin{document}

\begin{abstract}
In this work, we investigate the existence of multiple positive solutions for a weakly coupled system of nonlinear elliptic equations governed by Pucci’s extremal operators. Specifically, we consider the system:

\[
\begin{cases}
-\mathcal{M}_{\lambda_1,\Lambda_1}^+(D^2u_1) = \mu f_1(u_1, u_2, \dots, u_n), & \text{in } \Omega, \\
-\mathcal{M}_{\lambda_2,\Lambda_2}^+(D^2u_2) = \mu f_2(u_1, u_2, \dots, u_n), & \text{in } \Omega, \\
  \vdots \\
-\mathcal{M}_{\lambda_n,\Lambda_n}^+(D^2u_n) = \mu f_n(u_1, u_2, \dots, u_n), & \text{in } \Omega, \\
u_1 = u_2 = \dots = u_n = 0, & \text{on } \partial\Omega,
\end{cases}
\]

where \( \mathcal{M}_{\lambda,\Lambda}^+ \) represents the Pucci extremal operator, \( \Omega \) is a bounded domain in \( \mathbb{R}^N \) with smooth boundary, and the nonlinear functions \( f_i: [0, \infty)^n \to [0, \infty) \) belong to the \( C^{1,\alpha} \) class. Our main results establish the existence and multiplicity of solutions for sufficiently large values of the parameter \( \mu > 0 \). The analysis relies on the method of sub and supersolutions, in conjunction with fixed-point arguments and bifurcation techniques.
\end{abstract}

\maketitle


\section{Introduction}
The study of existence and multiplicity of solutions for the system:
\begin{equation}\label{problem1}
\left\{
\begin{aligned}
-\mathcal{M}_{\lambda_1,\Lambda_1}^+(D^2u_1) &= \mu((u_1)^{1-\alpha_1}+f_1(u_2)) &&\text{in } \Omega,\\
-\mathcal{M}_{\lambda_2,\Lambda_2}^+(D^2u_2) &= \mu((u_2)^{1-\alpha_2}+f_2(u_3)) &&\text{in } \Omega,\\
& \vdots \\
-\mathcal{M}_{\lambda_n,\Lambda_n}^+(D^2u_n) &= \mu((u_n)^{1-\alpha_n}+f_n(u_1)) &&\text{in } \Omega,\\
u_1 = u_2 = \dots = u_n &= 0 &&\text{on } \partial\Omega,
\end{aligned}
\right.
\end{equation}
was discussed in \cite{mallick2024multiplicity} using the three-solution theorem. Here, \( \Omega \) is a smooth and bounded domain in \( \mathbb{R}^N \), \( \mu \) is a positive parameter, \( \alpha_i \in (0,1) \), and \( f_i \in C^{\alpha}[0,\infty) \) for \( i=1,2,\dots,n \). The Pucci extremal operator is defined as:
\begin{equation}\label{pucci}
\mathcal{M}^{\pm}_{\lambda,\Lambda}(M)=\Lambda\sum_{\pm e_{i}>0}e_{i}+\lambda\sum_{\pm e_{i}<0}e_{i},
\end{equation}
where \( M \) is an \( N\times N \) symmetric matrix and \( e_i \) are its eigenvalues.

Extending this study to a general function \( f_i(u_1,u_2,\dots,u_n) \) of sublinear growth, we consider the system \eqref{problem3}.
\begin{equation}\label{problem3}
\left\{
\begin{aligned}
-\mathcal{M}_{\lambda_1,\Lambda_1}^+(D^2u_1) &= \mu f_1(u_1, u_2, \dots, u_n) &&\text{in } \Omega,\\
-\mathcal{M}_{\lambda_2,\Lambda_2}^+(D^2u_2) &=\mu f_2(u_1, u_2, \dots, u_n) &&\text{in } \Omega,\\
& \vdots \\
-\mathcal{M}_{\lambda_n,\Lambda_n}^+(D^2u_n) &= \mu f_n(u_1, u_2, \dots, u_n) &&\text{in } \Omega,\\
u_1 = u_2 = \dots = u_n &= 0 &&\text{on } \partial\Omega,
\end{aligned}
\right.
\end{equation}
Our focus is on obtaining existence and multiplicity results for \eqref{problem3} involving non-divergence operators, utilizing the three-solution theorem. For similar work in the case of divergence form of operator one can see the work by \cite{mallick2018bifurcation,ali2007positive}. This paper is structured as follows: Section 2 presents preliminary results on Pucci's extremal operator. Section 3 establishes existence and multiplicity results by constructing two sub and supersolutions. Finally, Section 4 provides an example demonstrating at least three positive solutions for \eqref{problem3} under suitable conditions.

\section{Preliminaries}
\noindent In this section we collect definitions and results which will be used throughout the article. 

\begin{definition}[see \cite{MHP},\cite{ishii1991viscosity}]\label{viscosity}
A vector valued function $u=(u_1,u_2,...,u_n)\in C(\overline{\Omega})\times C(\overline{\Omega})\times\dots\times  C(\overline{\Omega})$ is defined a viscosity subsolution (resp., supersolution) of \eqref{problem3} if for any $\phi\in C^{2}(\Omega)$ such that $u_{j}-\phi$~({for some $j=1,2...,n$}) attains a maximum (resp., minimum) at some point $x_{0}\in \Omega$ then we have
\begin{equation*}
-\mathcal{M}_{\lambda_{j},\Lambda_{j}}^+ (D^2\phi(x_{0}))\leq\mu f_j(u_{1}(x_{0}),u_{2}(x_{0}),\dots,u_{n}(x_{0}))~~\Big(\text{resp.,}~~\geq \mu f_j(u_{1}(x_{0}),u_{2}(x_{0}),\dots,u_{n}(x_{0}))\Big).
\end{equation*}
$u$ is solution if it is both subsolution and supersolution at the same time. Moreover, $u$ is called strict subsolution (resp., supersolution) of \eqref{problem3} if $u$ is a subsolution (resp., supersolution) of
\begin{equation*}
\left\{
\begin{aligned}
-\mathcal{M}_{\lambda_1,\Lambda_1}^+(D^2u_1)&=\mu f_1(u_1,u_2,\dots,u_n)+h_1~~~&&\rm{in}~~\Omega,\\
-\mathcal{M}_{\lambda_2,\Lambda_2}^+(D^2u_2)&=\mu f_2(u_1,u_2,\dots,u_n)+h_2~~~&&\rm{in}~~\Omega,\\
~~~~~~~~\vdots&=~~~~~~~~~~~~ \vdots\\
-\mathcal{M}_{\lambda_n,\Lambda_n}^+(D^2u_n)&=\mu f_n(u_1,u_2,\dots,u_n)+h_{n}~~~&&\rm{in}~~\Omega,
\end{aligned}
\right.
\end{equation*}
for some $h=(h_1,h_2,...,h_n)\in C(\overline{\Omega})\times\dots\times  C(\overline{\Omega})$ such that $h_j(x)<0$ (resp.,~$h_{j}(x)>0$) on $\bar{\Omega}$ for all $j=1,2,...,n.$
\end{definition}

\noindent We also recall fundamental eigenvalue results for Pucci’s operator:
\begin{theorem}[Proposition 1.1\cite{MR2124162}]\label{Pr2.5}
There exist a positive constant~$\mu^{+}_{1,i}$ and a function $\phi^{+}_{1,i}\in C^{2}(\Omega)\cup C(\bar{\Omega})$ such that:\\
\begin{equation}\label{28}
\left\{
\begin{aligned}{}
-\mathcal{M}^{+}_{\lambda_i,\Lambda_i}(D^{2}\phi^{+}_{1,i})&=\mu^{+}_{1,i} \phi^{+}_{1,i}~~\text{in}~\Omega\\
\phi^{+}_{1,i}&=0~\text{on}~\partial\Omega.
\end{aligned}
\right.
\end{equation}
Furthermore,~$\phi^{+}_{1,i}>0$ in $\Omega.$
\end{theorem}
\begin{theorem}[Theorem 17.18\cite{GT}]\label{Pr2.6}
Let $e_i\in C^2(\Omega)\cap C(\bar{\Omega})$ be the unique solution to the following problem
\begin{equation}\label{ee}
\left\{
\begin{aligned}
-\mathcal{M}_{\lambda_i,\Lambda_i}^+ (D^2e_i)&=1~~\mbox{in} ~~~\Omega;\\
e_i&=0~\mbox{on}~~\partial \Omega.
\end{aligned}
\right.
\end{equation}
\end{theorem}
\noindent It is evident that the solution $e_i$ satisfies $e_i\geq0.$ 

 \begin{theorem}[Theorem 3.5 \cite{mallick2024multiplicity}]\label{Th3.5}
Let $ \psi=( \psi_1, \psi_2,\dots, \psi_n)$ and $\phi=(\phi_1,\phi_2,\dots,\phi_n)$   be positive sub- and supersolutions of \eqref{problem3} satisfying \( \psi \leq \phi \) (i.e. $\psi_i\leq\phi_i$ for $i=1,2,\dots,n$). Then, there exists a minimal and a maximal solution in the order interval \( [\psi, \phi] \).
\end{theorem}
\begin{theorem}[Theorem 3.7 \cite{mallick2024multiplicity}, Three-Solution Theorem]\label{Th2.5}
Let $\psi^i=(\psi^i_1,\psi^i_2,\dots,\psi^i_n)$ and  $\phi^i=(\phi^i_1,\phi^i_2,\dots,\phi^i_n)$ for $i=1,2.$
Suppose there exist two ordered pairs of sub- and supersolutions \( (\psi^1,\phi^1) \) and \( (\psi^2,\phi^2) \) for \eqref{problem3}, satisfying:
\( \psi^1 \leq \psi^2 \leq \phi^1 \),
     \( \psi^1 \leq \phi^2 \leq \phi^1 \),
     \( \psi^2 \not\leq \phi^2 \), and
\( \psi^2, \phi^2 \) are not solutions of \eqref{problem3}.
Then, \eqref{problem3} admits at least three distinct solutions \( u^1, u^2, u^3 \), where:
\begin{equation*}
 u^1 \in [\psi^1, \phi^2], \quad u^2 \in [\psi^2, \phi^1], \quad \text{and} \quad u^3 \in [\psi^1, \phi^1] \setminus ([\psi^1, \phi^2] \cup [\psi^2, \phi^1]).
\end{equation*}
where $u^i=(u^i_1,u^i_2,\dots,u^i_n)$ for $i=1,2,3.$
\end{theorem}
\section{Main Results}
This section establishes the existence of positive solutions for \( \mu > 0 \) and multiple (three) positive solutions for a certain range of \( \mu \) for \eqref{problem3}. These results are obtained using the three-solution theorem (Theorem \ref{Th2.5}) from the previous section.

To prove our first result, we assume the following hypotheses:
\begin{itemize}
    \item[(C1)] \( f_{i} \) is non-decreasing in each variable, with \( f_{i}(0,0,\dots,0) = 0 \) for \( i = 1,2,\dots,n \).
    \item[(C2)] \( \displaystyle{\lim_{x_i \to 0} \frac{\partial f_{i}(x_1,x_2,\dots,x_i,\dots,x_n)}{\partial x_i} = \infty} \) for \( i=1,2,\dots,n \).
    \item[(C3)] \( \displaystyle{\lim_{s \to \infty} \frac{f_{i}(s,s,\dots,s)}{s} = 0} \) for \( i=1,2,\dots,n \).
\end{itemize}

\begin{theorem}\label{Th3.1}
Let conditions (C1) and (C2) be satisfied. Then, there exists  \( \mu_{0} > 0 \) such that for every \( 0 < \mu < \mu_0 \), system \eqref{problem3} admits a positive classical solution \( u = (u_1, u_2, \dots, u_n) \). Furthermore, \( \displaystyle{\sum_{i=1}^n \|u_i\|_{L^{\infty}(\Omega)} \to 0} \) as \( \mu \to 0 \).
\end{theorem}

\begin{proof}
Let \( e_i \in C^2(\Omega) \cap C(\overline{\Omega}) \) be the solution of  \eqref{ee} from Theorem \ref{Pr2.6}. Since \( f_i(0,0,\dots,0) = 0 \), we choose sufficiently small \( \mu^i_{1}, \mu^i_{2}, \dots, \mu^i_{n} > 0 \) such that for \( i=1,2,\dots,n \),
\begin{equation*}
f_i(\mu^i_{1}\|e_{1}\|_{\infty},\mu^i_{2}\|e_{2}\|_{\infty},\dots,\mu^i_{n}\|e_{n}\|_{\infty}) < 1.
\end{equation*}
Setting \( \mu_{0} = \displaystyle{\min_{i=1,2,\dots,n}}\{ \mu^i_{1}, \mu^i_{2}, \dots, \mu^i_{n} \} \), it follows that
\begin{equation*}
f_i(\mu_{0}\|e_{1}\|_{\infty},\mu_{0}\|e_{2}\|_{\infty},\dots,\mu_{0}\|e_{n}\|_{\infty}) < 1.
\end{equation*}
For a fixed \( \mu < \mu_{0} \), we define \( \phi_i = \mu e_i \). Then,
\[
-\mathcal{M}_{\lambda_i,\Lambda_i}^+ (D^2\phi_i) = \mu.1
\]
Using the monotonicity of \( f_i \), we get
\begin{equation*}
-\mathcal{M}_{\lambda_i,\Lambda_i}^+ (D^2\phi_i) > \mu f_i(\phi_1, \phi_2, \dots, \phi_n),
\end{equation*}
proving that \( \phi \) is a supersolution of \eqref{problem3}. 

A subsolution \( \psi \) is constructed using the principal eigenfunction \( \phi^+_{1,i} \) from Theorem \ref{Pr2.5}. Setting \( \psi_i = m_{\mu} \phi^+_{1,i} \),\\ 
since $\frac{\partial f_i(x_1,x_2,\dots,x_i,\dots,x_n)}{\partial x_i}\longrightarrow \infty$ as ${x_i \to 0}$, there exists $m_i=m_i(\mu)>0$ such that 
\begin{equation*}
{f_i(0,0,\dots,s,\dots,0)>\frac{\mu_{1,i}^+}{\mu}s }~~ \text{for every}~s \in (0,m_i),~~ \text{for}~~ i=1,2,\dots,n.
\end{equation*}
choose $m_{\mu}=$min$\{m_1,m_2,\dots,m_n\},$ we get\\
\begin{align*}
{f_i(0,0,\dots,s,\dots,0)>\frac{\mu_{1,i}^+}{\mu}s }~~ \text{for every}~s \in (0,m_{\mu}).\\
\end{align*}
we obtain
\[
-\mathcal{M}_{\lambda_i,\Lambda_i}^+ (D^2\psi_i) \leq \mu f_i(\psi_1, \psi_2, \dots, \psi_n).
\]
By Theorem \ref{Th3.5}, a solution \( (u_1, u_2, \dots, u_n) \) exists with \( \psi \leq u \leq \phi \), and \( \sum \|u_i\|_{\infty} \to 0 \) as \( \mu \to 0 \).
\end{proof}

The next theorem extends this result for all \( \mu > 0 \), assuming that \( f_i \) satisfies (C3), which ensures that the growth of \( f_i \) at infinity remains controlled and does not exceed a linear rate.

\begin{theorem}\label{Th3.2}
Let conditions $(C1)-(C3)$ be  satisfied. Then, the system \eqref{problem3} admits a positive solution $u$ for every $\mu>0.$
\end{theorem}

\begin{proof}
Let $(\psi_1,\psi_2,\dots,\psi_n)$ be the subsolution obtained in Theorem \ref{Th3.1}. Now, we proceed to construct a supersolution $\phi=(\phi_1,\phi_2,\dots,\phi_n)$ of \eqref{problem3}. 
Utilizing conditions $(C1)$ and $(C3)$, we can choose $\tilde{m}_{\mu} \gg 1$ for $\mu>0$ such that for $i=1,2,\dots,n$, it holds that
\begin{equation*}
\dfrac{1}{\mu\|e_i\|_{\infty}}\geq \dfrac{f_{i}(\tilde{m}_{\mu}\|e_1\|_{\infty},\tilde{m}_{\mu}\|e_2\|_{\infty},\dots,\tilde{m}_{\mu}\|e_n\|_{\infty})}{\tilde{m}_{\mu}\|e_i\|_{\infty}}
\end{equation*}
For $i=1,2,\dots,n$, we set $\phi_i=\tilde{m}_{\mu} e_i$. Then,
\begin{align*}
-\mathcal{M}_{\lambda_i,\Lambda_i}^+ (D^2{\phi}_i)=-\mathcal{M}_{\lambda_i,\Lambda_i}^+ (D^2\tilde{m}_{\mu} e_i)
 =&\tilde{m}_{\mu}\\
\geq&\mu\ f_i(\tilde{m}_{\mu}\|e_{1}\|_{\infty},\tilde{m}_{\mu}\|e_{2}\|_{\infty},\dots,\tilde{m}_{\mu}\|e_{n}\|_{\infty})\\
\geq&\mu \ f_i(\tilde{m}_{\mu} e_{1},\tilde{m}_{\mu} e_{2},\dots,\tilde{m}_{\mu} e_{n})\\
=&\mu \ f_i(\phi_1,\phi_2,\dots,\phi_n)
\end{align*}
Thus, ${\phi}=({\phi}_1,{\phi}_2,\dots,{\phi}_n)$ is a supersolution of \eqref{problem3} for all $\mu>0.$ Furthermore, since $\dfrac{\partial e_i}{\partial n}<0$ and $\tilde{m}_\mu\gg1,$ we can select $m_{\mu}$ sufficiently small to ensure that $\psi_i\leq \phi_i$ for all $i=1,2,\dots,n$. Therefore, for each $\mu>0,$ Problem \ref{problem3} has a solution $u=(u_1,u_2,\dots,u_n)$ satisfying $\psi\leq u\leq{\phi}.$
\end{proof}

\noindent Our next result(Theorem \eqref{Th3.3}) deals with the existence of multiple(three) solutions to our problem \eqref{problem3} for certain range of values of $\mu.$ In order to prove this we require the following additional condition of $f_i$'s.
\begin{itemize}
\item[(C4)] Suppose there exist two positive constants $0<a<b$ with $f_i(a,a,\dots,a)\neq 0$
and $f_i(b,b,\dots,b)\neq 0$ for $i=1,2,\dots n$ and $\displaystyle{\min_{i=1,2,\dots,n}}\left\{\dfrac{1}{\|e_i\|_{\infty}}\dfrac{a}{f_i(a,a,\dots,a)}\right\}>\displaystyle{\max_{i=1,2\dots,n}}\left\{A_i
\dfrac{b}{f_i(b,b,\dots,b)}\right\},$
\end{itemize}
where $A_i=\displaystyle{\inf_{\epsilon}}\dfrac{N_i^{-}R^{N_i^+-1}}{\epsilon^{N_i^{-}}(R-\epsilon)}$, $N_i^-=\frac{\Lambda_i}{\lambda^i}(N-1)+1$, $N_i^+=\frac{\lambda_i}{\Lambda_i}(N-1)+1$ and $R$ is the radius of the largest inscribed ball $B_R$ in $\Omega.$
\begin{theorem}\label{Th3.3}
Let $f_i$'s satisfy $(C1)$-$(C4).$ Then there exist  $0\leq \mu_*<\mu^*$ such that for any $\mu\in[\mu_{*},\mu^{*}],$ system \eqref{problem3} has at least three solutions.
\end{theorem}
\begin{proof}
By Theorem \ref{Th2.5}, the existence of three solutions follows once we establish two pairs of subsolutions and supersolutions for system \eqref{problem3} that are appropriately ordered . Here, we consider $\Omega=B_{R}.$ Let $\psi$ and ${\phi}$ be as defined in Theorem \ref{Th3.2}.
Now, we proceed to construct a strict supersolution $\Phi$ to \eqref{problem3}. Let us set $\tilde{\phi}=(\tilde{\phi}_1,\tilde{\phi}_2,\dots,\tilde{\phi}_n)=\left(\dfrac{ae_1}{\|e_1\|_\infty},\dfrac{ae_2}{\|e_2\|_\infty},\dots,\dfrac{ae_n}{\|e_n\|_\infty}\right)$
where $e_i$'s are as above. Thus, for $\mu<\mu^{*}=\displaystyle{\min_{i=1,2,\dots,n}}\left\{\dfrac{1}{\|e_i\|_{\infty}}\dfrac{a}{f_i(a,a,\dots,a)}\right\}$,
and $i=1,2,\ldots,n,$ we have
\begin{align*}
-\mathcal{M}_{\lambda_i,\Lambda_i}^+ (D^2\tilde{\phi}_i)&=\frac{a}{\|e_i\|_\infty}\\
&>\mu^*  f_i(a,a,\dots,a)\\
&\geq\mu f_i(a,a,\dots,a)\\
&\geq\mu f_i\left(\dfrac{ae_1}{\|e_1\|_\infty},\dfrac{ae_2}{\|e_2\|_\infty},\dots,\dfrac{ae_n}{\|e_n\|_\infty}\right)\\
&=\mu f_i(\tilde{\phi}_1,\tilde{\phi}_2,\dots,\tilde{\phi}_n)
\end{align*}
Thus, $\tilde{\phi}$ is a strict supersolution of (\ref{problem3}) for $\mu<\mu^{*}.$ Now, we proceed to construct a positive strict subsolution $\tilde{\psi}$ of (\ref{problem3}) for $\mu>\mu_{*}=\displaystyle{\max_{i=1,2\dots,n}\left\{A_i
\dfrac{b}{f_i(b,b,\dots,b)}\right\}}$. For $0<\epsilon<R;~l,m>1,$ define $\rho:[0,R]\to[0,1]$ as
\[ \rho(r)=\left\{
\begin{array}{l l}
1 & \quad \text{$0 \leq r \leq \epsilon,$}\\
1-\left(1-\left(\dfrac{R-r}{R-\epsilon}\right)^m\right)^l &\quad \text{$\epsilon < r \leq R.$}
\end{array}\right.
\]
and set $d(r)=b\rho(r).$ Then
\[ \rho^{\prime}(r)=\left\{
\begin{array}{l l}
0 & \quad \text{$0 \leq r \leq \epsilon,$}\\
-\dfrac{lm}{R-\epsilon}\left(1-\left(\dfrac{R-r}{R-\epsilon}\right)^m\right)^{l-1}\left(\dfrac{R-r}{R-\epsilon}\right)^{m-1} &\quad \text{$\epsilon < r \leq R.$}
\end{array}\right.
\]
 we obtain
\begin{equation}\label{10}
|\rho^{\prime}(r)| \leq \dfrac{lm}{R-\epsilon}~~\text{and}~|d^{\prime}(r)| \leq \dfrac{b lm}{R-\epsilon}
\end{equation}
Let us examine the solution of the following system:
\begin{equation}\label{211}
\left\{
\begin{split}
-\mathcal{M}_{\lambda_1,\Lambda_1}^+ (D^2\tilde{\psi}_1)&=\mu f_1(d(r),d(r),\dots,d(r))~~&&\text{in}~~ B_{R},\\
-\mathcal{M}_{\lambda_2,\Lambda_2}^+ (D^2\tilde{\psi}_2)&=\mu f_2(d(r),d(r),\dots,d(r))~~&&\text{in}~~ B_{R},\\
&\vdots=\vdots\\
-\mathcal{M}_{\lambda_n,\Lambda_n}^+ (D^2\tilde{\psi}_n)&=\mu f_n(d(r),d(r),\dots,d(r))~~&&\text{in}~~ B_{R},\\
\tilde{\psi}_1=\tilde{\psi}_2=\dots=\tilde{\psi}_n&=0~~&&\text{on}~~\partial B_{R},
\end{split}
\right.
\end{equation}
and set $\tilde{\psi}=(\tilde{\psi}_1,\tilde{\psi}_2,\dots,\tilde{\psi}_n).$ 
By well known regularity results the function $\tilde{\psi_i}\in C^{2}({B_{R}})\cap C(\partial B_R),$ for $i=1,2,\dots, n$.  It follows from Theorem 1.1\cite{MR4097925} that $\tilde{\psi}_i$ is radially symmetric. As a result, we obtain:  
\begin{equation}\label{12}
\left\{
\begin{split}
-&\theta_1(\tilde{\psi}_1''(r))\tilde{\psi}_1''(r)-\frac{N-1}{r}\theta_1(\tilde{\psi}_1'(r))\tilde{\psi}_1'(r)= \mu f_1(d(r),d(r),\dots,d(r)) ~&&r\in (0,R);\\
-&\theta_2(\tilde{\psi}_2''(r))\tilde{\psi}_2''(r)-\frac{N-1}{r}\theta_2(\tilde{\psi}_2'(r))\tilde{\psi}_2'(r)= \mu f_2(d(r),d(r),\dots,d(r)) ~&&r\in (0,R);\\
&\vdots&&\vdots\\
-&\theta_n(\tilde{\psi}_n''(r))\tilde{\psi}_n''(r)-\frac{N-1}{r}\theta_n(\tilde{\psi}_n'(r))\tilde{\psi}_n'(r)= \mu f_n(d(r),d(r),\dots,d(r)) ~&&r\in (0,R);\\
&\tilde{\psi}_1'(0)=\tilde{\psi}_2'(0)=\dots=\tilde{\psi}_n'(0)=0;\\
&\tilde{\psi}_1(R)=\tilde{\psi}_2(R)=\dots=\tilde{\psi}_n(R)=0.
\end{split}
\right.
\end{equation}
where
\[
    \theta_i(s)=
\begin{cases}
    \Lambda_i,& \text{if}~s\geq0\\
    \lambda_i,& \text{if}~s<0.
\end{cases}
\]
For a fixed $1\leq i\leq n,$ we rewrite the $i$th equation in \eqref{12} as follows:
\begin{align}\label{3.1}
\tilde{\psi}_i''(r)+\frac{(N-1)\theta_{i}(\tilde{\psi}_i'(r))}{r\theta_i(\tilde{\psi}_i''(r))}\tilde{\psi}_i'(r)=-\frac{\mu}{\theta_i(\tilde{\psi}_i''(r))}f_i(d(r),d(r),\dots,d(r)).
\end{align}
Furthermore, by setting $v_i(r)=\frac{\theta_i(\tilde{\psi}_i'(r))(N-1)}{\theta_i(\tilde{\psi}_i''(r))r}$, $\tau_i(r)=\exp{{\int_1^rv_i(s)ds}}$ and $\tilde{\tau_i}(r)=\frac{\tau_i(r)}{\theta_i(\tilde{\psi}_i''(r))},$ Equation \ref{3.1} can further be rewritten as follows:
\begin{equation}\label{3.2}
\begin{aligned}{}
\tau_i(r)\tilde{\psi}_i''(r)+\tau_{i}(r)\frac{\theta_{i}(\tilde{\psi}_i'(r))}{\theta_{i}(\tilde{\psi}_i''(r))}\frac{N-1}{r}\tilde{\psi}_i'(r)&=-\mu\tilde{\tau_i}(r)f_i(d(r),d(r),\dots,d(r)),\\
\left(\tau_i(r)\tilde{\psi}_i'(r)\right)'&=-\mu\tilde{\tau_{i}}(r)f_{i}(d(r),d(r),\dots,d(r)),
\end{aligned}
\end{equation}
Integrating \eqref{3.2}  from $0$ to $r$ we have
\begin{align}
\int_0^r\left(\tau_i(s)\tilde{\psi}_i'(s)\right)'ds=-\mu\int_0^r\tilde{\tau_i}(s)f_i(d(s),d(s),\dots,d(s))ds
\end{align}
which implies
\begin{align*}
\tilde{\psi}_i'(r)=-\frac{\mu}{\tau_i(r)}\int_0^r\tilde{\tau_i}(r)f_i(d(s),d(s),\dots,d(s))ds~~~(\text{since}~~\Psi'_i(0)=0).
\end{align*}
For each fixed $1\leq i\leq n,$ we assert that for $\mu>\mu_{*},$ $\tilde{\psi}_i(t)>d(t)$ for all $t\in[0,R).$ If this assertion holds then

$$-\mathcal{M}_{\lambda_i,\Lambda_i}^+(D^2\tilde{\psi}_i)=\mu f_i(d(r),d(r),\dots,d(r))<\mu f_i(\tilde{\psi}_1(r),\tilde{\psi}_2(r),\dots,\tilde{\psi}_n(r))~~~\rm{in}~~B_R,$$
that is, $(\tilde{\psi}_1,\tilde{\psi}_2,\dots,\tilde{\psi}_n)$ is a strict supersolution of \eqref{problem3}.\\
{\it{Proof of claim:}}  To establish that $d(t)<\tilde{\psi}_i(t),$ it suffices to show $\tilde{\psi}_i'(t)<d'(t)$ on $(0,R],$ since $\tilde{\psi}_i(R)=d(R)=0.$
For any $r\in (0,\epsilon)$, $f_i(d(r),d(r),\dots,d(r))=f_i(b,b,\dots,b)$ and $\tilde{\tau_i}(r)>0,$ which implies $\tilde{\psi}_i'(r)<0=d'(r)$ confirming the claim . To extend this proof  to $r\in(\epsilon,R).$ We observe that
\begin{itemize}
\item[(i)]~$N_{i}^+-1\leq v_i(r)r\leq N_{i}^--1,$
\item[(ii)]~$r^{N_{i}^--1}\leq\tau_i(r)\leq r^{N_{i}^+-1}$~and~$\frac{\tau_i(r)}{\Lambda_i}\leq\tilde{\tau_i}(r)\leq\frac{\tau_i(r)}{\lambda_i}.$
\end{itemize}
Now consider
\begin{align}\nonumber
-\tilde{\psi}_i'(r)=&\frac{\mu}{\tau_i(r)}\int_0^r\tilde{\tau_i}(s)f_i(d(s),d(s),\dots,d(s))ds\geq\frac{\mu}{\tau_i(r)}\int_0^{\epsilon}\tilde{\tau}_i(s)f_i(d(s),d(s),\dots,d(s))ds~~~~(\text{since}~r>\epsilon)\\ \nonumber
\geq& \frac{\mu}{r^{N_{i}^+-1}}\int_0^\epsilon\tilde{\tau_i}(s)f_i(d(s),d(s),\dots,d(s))ds\geq\frac{\mu}{\Lambda_i r^{N_{i}^+-1}}\int_0^{\epsilon}s^{N_{i}^--1}f_i(b,b,\dots,b)ds\nonumber\\
\geq&\frac{\mu f_i(b,b,\dots,b)}{\Lambda_i R^{N_{i}^+-1}}\int_0^{\epsilon}s^{N_{i}^--1}ds=\frac{\mu f_i(b,b,\dots,b)\epsilon^{N_{+}^-}}{\beta R^{N_{i}^+-1}N_{i}^-}.\label{4.9}
\end{align}
Thus, in order to show $-\tilde{\psi}_i'(r)< -d'(r),$  in view of \eqref{4.9} and \eqref{10}, it is sufficient to show that
\begin{equation}\label{10.1}
\frac{\mu f_i(b,b,\dots,b)\epsilon^{N^{-}_i}}{\Lambda_i R^{N_{+}^i-1}N^{-}_i}>\frac{lmb}{R-\epsilon}.
\end{equation}
Let us take $\epsilon_i=\frac{N_{i}^-R}{N_{i}^-+1}$ at which the function $\frac{1}{(R-\epsilon)\epsilon^{N_{i}^-}}$ has minimum for each $i=1,2,\dots,n$. From $(C5)$ we have  $\mu>\displaystyle{\max_{i=1,2,\dots,n}\left\{\frac{b}{f_i(b,b,\dots,b)}\frac{\Lambda_i N_{i}^-R^{N_{i}^+-1}}{(R-\epsilon_i)\epsilon_i{^{N_{i}^-}}}\right\}}=\mu_*$. Hence we can choose $l,m>1$ such that
\begin{equation}\label{3.4}
\mu>\frac{lmb\Lambda_i N_{i}^-R^{N_{i}^+-1}}{f_i(b,b,\dots,b)(R-\epsilon_i)\epsilon_i^{N_{i}^-}}.
\end{equation}
holds for $i=1,2,\dots,n$,
confirming that $\tilde{\psi}$ is a  strict subsolution for $\mu>\mu_*$. Since $\tilde{\psi}_i(0)>d(0)=b\rho(0)=b>a=\|\tilde{\phi}\|_{\infty}$, it follows that $\tilde{\psi}_i\nleq\tilde{\phi}_i$. By Theorem \ref{Th3.2} we can select a sufficiently small  subsolution $\psi$ and a sufficiently large supersolution $\phi$ ensuring that
 $\psi\leq \tilde{\psi}\leq \phi$, $\psi\leq\tilde{\phi}\leq \phi$. Therefore, applying Theorem \ref{Th2.5} we conclude that there exist at least three positive solutions for $\mu\in(\mu_*,\mu^*)$.\\

\noindent Now, for a general bounded domain $\Omega$, let $B_R$ be the largest inscribed ball within $\Omega$. To establish the existence of three positive solution we utilize the sub and supersolutions constructed in $B_R$ and extend them by zero in $\Omega\setminus B_R$ (ensuring regularity). It is straightforward to verify  that this approach yields  at least three positive solutions for $\mu\in(\mu_*,\mu^*)$.
\end{proof}
\section{Example}
\noindent Here, we present an example that meets the assumptions of Theorems \ref{Th3.1}-\ref{Th3.2}. when the operator is the Laplacian, the following  problem appears in combustion theory and has been extensively studied by various authors e.g. Parks\cite{PJR}. Consider
\begin{equation}\label{exp}
\left\{
\begin{split}
-\mathcal{M}_{\lambda_1,\Lambda_1}^+(D^2u_1)&=\mu\left(u_1^{\alpha_1}+e^{\frac{\tau u_2}{\tau+u_2}}-1+u_3^{\alpha_3}\dots+u_n^{\alpha_n}\right);&&~~~\rm{in}~\Omega ,\\
-\mathcal{M}_{\lambda_2,\Lambda_2}^+(D^2u_2)&=\mu\left(u_1^{\alpha_1}+u_2^{\alpha_2}+e^{\frac{\tau u_3}{\tau+u_3}}-1+\dots+u_n^{\alpha_n}\right);&&~~~\rm{in}~\Omega,\\
\vdots&\vdots\\
-\mathcal{M}_{\lambda_n,\Lambda_n}^+(D^2u_n)&=\mu\left(e^{\frac{\tau u_1}{\tau+u_1}}-1+u_2^{\alpha_2}+u_3^{\alpha_3}\dots+u_n^{\alpha_n}\right);&&~~~\rm{in}~\Omega,\\
u_1=u_2=\dots&=0;&&~~~\rm{on}~~\partial\Omega,
\end{split}
\right.
\end{equation}
where $\alpha_i\in(0,1)$ for $i=1,2,\dots,n$, $\tau>0$.Here the function  $f_i(u_1,u_2,\dots,u_n)$ is given by 
$f_i(u_1,u_2,\dots,u_n)=e^{\frac{\tau u_i}{\tau+u_i}}-1+\sum_{j\neq i} u_j^{\alpha_j} $ for $i=1,2,\dots,n$ and $j=1,2,\dots,n$ , it is evident that $f_i(0,0,\dots,0)=0,$ the derivatives $f_i's $ are non-decreasing and the function $f$ exhibits  sublinear growth at infinity.Consequently it satisfies the conditions (C1),(C2) \& (C3), ensuring the applicability of  Theorem \ref{Th3.1}-\ref{Th3.2}. Furthermore, for sufficiently large  $\tau\gg 1,$ by choosing $a=1$, $b=\tau$, we obtain $\mu^*=\displaystyle{\min_{i=1,2,\dots,n}\left\{\frac{1}{\|e_i\|_{\infty}}\left\{\frac{a} {f_i(a,a,\dots,a)}\right\}\right\}}=\frac{1}{(n-1)}\displaystyle{\min_{1,2,\dots,n}\left\{\frac{1}{\|e_i\|_{\infty}}\right\}}$, \\ and
$\mu_*=\displaystyle{\max_{1,2,\dots,n}\left\{A_i\frac{\tau}{f_i (\tau,\tau,\dots,\tau)}\right\}=\max\left\{\frac{1}{\frac {e^{0.5\tau }-1}{\tau}+\Sigma_{j\neq i} 
 \tau^{\alpha_j}}A_i\right\}}\to 0$,. Consequently, $\mu_*<\mu^*$ for $\tau\gg 1$. Thus, the functions  $f_i$'s in \eqref{exp} also satisfies (C4).Consequently,  for $\mu\in (\mu_*,\mu^*),$ Theorem \ref{Th3.3} guarantees that \eqref{exp} has at least three positive solutions.
\bibliography{MultiGenSys.bib}
\bibliographystyle{abbrv}

\end{document}